\documentclass{amsart}
\usepackage{amssymb}
\usepackage{amsfonts}
\usepackage{amssymb}
\usepackage{amsmath}
\usepackage{amsthm}
\usepackage{enumerate}
\usepackage{tabularx}
\usepackage{centernot}
\usepackage{mathtools}
\usepackage{stmaryrd}
\usepackage{amsthm,amssymb}
\usepackage{etoolbox}
\usepackage{tablefootnote}
\usepackage{caption} 

\usepackage{tikz}
\usetikzlibrary{decorations.pathmorphing,shapes}

\newcounter{sarrow}

\renewcommand{\leq}{\leqslant}
\renewcommand{\geq}{\geqslant}
\renewcommand{\nleq}{\nleqslant}

\makeatletter
\def\subsection{\@startsection{subsection}{3}%
  \z@{.5\linespacing\@plus.7\linespacing}{.3\linespacing}%
  {\bfseries\centering}}
\makeatother

\makeatletter
\def\subsubsection{\@startsection{subsubsection}{3}%
  \z@{.5\linespacing\@plus.7\linespacing}{.3\linespacing}%
  {\centering}}
\makeatother

\makeatletter
\def\myfnt{\ifx\protect\@typeset@protect\expandafter\footnote\else\expandafter\@gobble\fi}
\makeatother

\theoremstyle{definition}

\newtheorem{theorem}{Theorem}[section]
\newtheorem{definition}[theorem]{Definition}
\newtheorem{lemma}[theorem]{Lemma}
\newtheorem{proposition}[theorem]{Proposition}

\newtheorem{corollary}[theorem]{Corollary}

\newtheorem{remark}[theorem]{Remark}
\newtheorem{notation}[theorem]{Notation}

\newcounter{claimcounter}
\numberwithin{claimcounter}{theorem}

\newcommand{\pureindep}[1][]{%
  \mathrel{
    \mathop{
      \vcenter{
        \hbox{\oalign{\noalign{\kern-.3ex}\hfil$\vert$\hfil\cr
              \noalign{\kern-.7ex}
              $\smile$\cr\noalign{\kern-.3ex}}}
      }
    }\displaylimits_{#1}
  }
}

\newcommand{\indep}[2]{%
  \mathrel{
    \mathop{
      \vcenter{
        \hbox{%
\oalign{
\noalign{\kern-.3ex}\hfil$\vert$\hfil\cr
              \noalign{\kern-.7ex}
              $\smile$\cr\noalign{\kern-.3ex}
}
}
      }
}^{\!\!\!\!\!#2}_{\!\!\hspace{-0.1em}#1}
  }
}

\newcommand{\displayindep}[2]{%
  \mathrel{
    \mathop{
      \vcenter{
        \hbox{%
\oalign{
\noalign{\kern-.3ex}\hfil$\vert$\hfil\cr
              \noalign{\kern-.7ex}
              $\smile$\cr\noalign{\kern-.3ex}
}
}
      }
}^{\!\!\hspace{-0.1em}#2}_{\!\!\hspace{-0.1em}#1}
  }
}

\newcommand{\displayfindep}[2]{%
  \mathrel{
    \mathop{
      \vcenter{
        \hbox{%
\oalign{
\noalign{\kern-.3ex}\hfil$\vert$\hfil\cr
              \noalign{\kern-.7ex}
              $\smile$\cr\noalign{\kern-.3ex} 
}
}
      }
}^{\!\hspace{-0.14em}#2}_{\!\!\hspace{-0.05em}#1}
  }
}

\newcommand{\divindep}[1][]{\indep{#1}{\mathrm{d}}}

\newcommand{\displaydivindep}[1][]{\displayindep{#1}{\mathrm{d}}\!\hspace{-0.15em}}

\newcommand*\dep{{=\mkern-1.2mu}}

\newcommand{\perpc}[1]{\perp_{#1}}

\def\presuper#1#2%
  {\mathop{}%
   \mathopen{\vphantom{#2}}^{#1}%
   \kern-\scriptspace%
   #2}

\begin{document}

\title[Reduction of Dat. Indep. to Dividing in Atomless Boolean Algebras]{Reduction of Database Independence to Dividing in Atomless Boolean Algebras}
\thanks{The research of the second author was supported by the Finnish Academy of Science and Letters (Vilho, Yrj\"o and Kalle V\"ais\"al\"a foundation).  The authors would like to thank the referee for suggesting a simpler and slightly more general proof of the main theorem. The authors would also like to thank John Baldwin and Jouko V\"a\"an\"anen for useful suggestions and discussions related to this paper.}

\author{Tapani Hyttinen}
\address{Department of Mathematics and Statistics,  University of Helsinki, Finland}

\author{Gianluca Paolini}
\address{Department of Mathematics and Statistics,  University of Helsinki, Finland}

\date{}
\maketitle

\begin{abstract} 
	
	We prove that the form of conditional independence at play in database theory and independence logic is reducible to the first-order dividing calculus in the theory of atomless Boolean algebras.
	This establishes connections between database independence and stochastic independence. Indeed, in light of the aforementioned reduction and recent work of Ben-Yaacov \cite{on_ran_var}, the former case of independence can be seen as the discrete version of the latter.
		
\end{abstract}


\section{Introduction}

At the current state of the art, the model-theoretic study of independence is a well-developed subject with numerous examples arising from different areas of mathematics. So-called independence calculi have been elaborated for several classes of structures, e.g. classes elementary in classical first-order logic, classes elementary in continuous logic, finitary abstract elementary classes, etc. Furthermore, these calculi are canonical, in the sense that the satisfaction of few crucial axioms imposes their uniqueness (for an overview see \cite{asa}). In Table \ref{indep} we see some examples. In all these examples the natural internal notion of independence coincides with the model-theoretic one.

	\begin{table}[h]
\begin{tabular}{c|c|c|c}
	\text{Class} & \text{Example} & \text{Independence} & \text{References} \\
	\hline
	\hline
	\rule{0pt}{3ex}  \text{Elementary in} & \text{Vector spaces} &  \text{Linear} & \text{\cite{shelah}} \\
	\text{classical $\mathrm{FOL}$} &  & \text{independence}  \\
		\hline
	\rule{0pt}{3ex}  \text{Elementary in} & \text{Hilbert spaces} &  \text{Orthogonality} & \text{\cite{CFOL_and_stab}} \\
	\text{continous logic} &  &  \text{}  \\
		\hline
	\rule{0pt}{3ex}  \text{Homogeneous} & \text{$\ell_p$-spaces} &  \text{Variations of} & \text{\cite{asa_and_tapani}} \\
	\text{metric} &  & \text{orthogonality} \\
	\text{classes} & & \\
		\hline
	\rule{0pt}{3ex}  \text{Quasi-minimal} & \text{Covers of} & \text{Independence} & \text{\cite{tapani_and_kaisa}}\\
	\text{pregeometry} & \text{smooth algebraic} & \text{from Zariski}\\
	\text{structures} &  \text{curves}    &  \text{topology}      \\
		\hline
	\rule{0pt}{3ex}  \text{Other abstract} & \text{$\mathbb{Z}_p$-modules} & \text{For closed pure} & \text{\cite{asa_and_tapani_padics}\tablefootnote{This is an $\mathrm{AEC}$ version of the metric class studied in \cite{asa_and_tapani_padics}. For more see \cite{baldwin}.}} \\
	\text{elementary} & \text{without elem.} &  submodules the usual \\
	\text{classes} & \text{of infinite height} & \text{amalgam of modules\tablefootnote{I.e. the sets are independent if and only if relative to each other the positions of the sets are as in the amalgam.}} & \\
		\hline
		\rule{0pt}{3ex} \text{Almost abstract} &  \text{Certain classes} &  \text{Free amalgam} & \text{\cite{geom_lat}} \\
	\text{elementary} &  \text{of geometric lattices} &  \text{of geometric}\\
	\text{classes\tablefootnote{I.e. classes of structures with a strong submodel relation which are not $\mathrm{AECs}$ but still behave well. E.g. in \cite{geom_lat} all the axioms of an $\mathrm{AEC}$ but the Smothness Axiom are satisfied,  but it is still
possible to define an independence notion along to the lines
of good frames.}} & \text{of fixed finite rank} &  \text{lattices$^2$} & \\
	\hline
	\hline
	\end{tabular}
\caption{Independence}\label{indep}\end{table}

	Another important case of independence which has been studied from the point of view of model theory is {\em stochastic independence}. In \cite{on_ran_var}, Itai Ben-Yaacov shows that the class of atomless probability algebras is elementary in continuous logic and that in this class non-forking corresponds to the familiar notion of independence of probability algebras, which in turn corresponds to independence of random variables, a.k.a. stochastic independence.
	 		
	In various papers on statistics and database theory, e.g. \cite{pearl}, \cite{sagiv} and \cite{studeny}, stochastic independence has been put in relation with a possibly less familiar notion of independence, known as embedded multivalued dependence or {\em database independence}, for simplicity. In these papers it is shown that these two forms of independence agree on a large group of axioms, which resemble very closely some basic axioms satisfied by non-forking in simple first-order theories, e.g. monotonicity, symmetry and transitivity. 
	
	We then ask: how does database independence relate to the general theory of independence as developed in model theory? Is there a way to make sense of the noticed similarities between stochastic independence and database independence from the point of view of model theory?
	
	In the following we answer positively both questions, showing that, in analogy with Ben Yaacov's work, database independence is reducible to first-order dividing calculus in the theory of atomless Boolean algebras. Our reduction shows that both database independence and stochastic independence are essentially based on the same independence relation, i.e. free amalgamation, considered in the category of Boolean algebras and probability algebras, respectively.

	The structure of the paper is as follows. In Section~\ref{preliminary_lemmas} we first give some basic definitions, and then introduce an independence relation between Boolean algebras and prove some preparatory lemmas. In Section~\ref{main_theorem} we prove our main theorem, which is a partial characterization of dividing in atomless Boolean algebra. In Section~\ref{sec_on_red} we then use this characterization to establish an abstract reduction of database independence to dividing in atomless Boolean algebras. In Section~\ref{emb_mult_dep} we introduce some basic definitions from database theory, formulate database independence, and make concrete the above mentioned reduction. Finally, we spend a few words on dependence logic, explaining why this form of independence can be identified with database independence and how the present work motivates probabilistic versions of independence logic, currently under development by the authors.

\section{Independence in Atomless Boolean Algebras}\label{preliminary_lemmas}

	For model-theoretic notations and conventions we refer to \cite{marker}. As usual with contemporary model theory, we work in large homogeneous models with a high degree of saturation, also known as {\em monster models}.

	\begin{definition} Let $T$ be a complete first-order theory with infinite models, and $\mathfrak{M}$ its monster model. For $a \in \mathfrak{M}^{n}$, $b \in \mathfrak{M}^m$ and $C \subseteq \mathfrak{M}$ we say that $a$ is {\em dividing independent} from $b$ over $C$, in symbols $a \divindep[C] b$, if for every $(b_i)_{i< \omega} \in (\mathfrak{M}^m)^{\omega}$ indiscernible over $C$ such that $b_0 = b$, there exists $a' \in \mathfrak{M}^n$ such that:
		\begin{enumerate}[i)]
			\item $\mathrm{tp}(a'/Cb) = \mathrm{tp}(a/Cb)$;
			\item $(b_i)_{i< \omega}$ is indiscernible over $Ca'$.
	\end{enumerate}	
\end{definition}

For notions and notation concerning Boolean algebras we refer to \cite{halmos}, for the model-theoretic aspects of the theory of Boolean algebras to \cite{poizat}. We denote by $\mathrm{BA}$ the theory of Boolean algebras in the signature $L = \left\{ 0, 1, \wedge, \vee, \neg \right\}$, with $\mathrm{ABA}$ the theory of atomless Boolean algebras and with $\mathrm{ATBA}$ the theory of atomic Boolean algebras.
	It is well-known that the theory $\mathrm{ABA}$ is complete and admits elimination of quantifiers. This theory is neither $\mathrm{NIP}$ nor $\mathrm{NSOP}$, and in particular it is not simple.	
 	Analogously, the theory $\mathrm{ATBA}$ is complete and admits elimination of quantifiers in the language extended with predicates $(A_i(x))_{i < \omega}$ saying that $x$ dominates at most $i$ atoms.

	The class of Boolean algebras with homomorphisms of Boolean algebras as morphisms is a category, we denote it by $\mathbf{BolAlg}$. Given $\mathcal{A} \models \mathrm{BA}$ and $\mathcal{B}$ a $\left\{ 0, 1, \wedge, \vee, \neg \right\}$-structure, we denote by $\mathcal{B} \leq \mathcal{A}$ the substructure relation (notice that $\mathrm{BA}$ is a $\forall$-theory and so its class of models is closed under taking substructures). For $\mathcal{A}, \mathcal{B} \models \mathrm{BA}$ and $\mathcal{C} \leq \mathcal{A}, \mathcal{B}$, we denote by $\mathcal{A} \otimes_{\mathcal{C}} \mathcal{B}$ the {\em push-out} of $\mathcal{A}$ and $\mathcal{B}$ over $\mathcal{C}$ in $\mathbf{BolAlg}$. The object $\mathcal{A} \otimes_{\mathcal{C}} \mathcal{B}$ can be shown to be the Boolean algebras corresponding to the tensor product of the Boolean rings corresponding to $\mathcal{A}$, $\mathcal{B}$ and $\mathcal{C}$, respectively.

	\begin{definition}\label{independence} Let $\mathcal{D} \models \mathrm{BA}$, $\mathcal{A}, \mathcal{B} \leq \mathcal{D}$ and $\mathcal{C} \leq \mathcal{A}, \mathcal{B}$. We let
		\[ \mathcal{A} \pureindep[\mathcal{C}] \mathcal{B} \;\; \Leftrightarrow \;\; \langle A \cup B \rangle \cong \mathcal{A} \otimes_{\mathcal{C}} \mathcal{B}.\]	
\end{definition}

	This structural definition of independence between Boolean algebras is equivalent to a more intuitive condition of topological nature.

	\begin{lemma} Let $\mathcal{D} \models \mathrm{BA}$, $\mathcal{A}, \mathcal{B} \leq \mathcal{D}$ and $\mathcal{C} \leq \mathcal{A}, \mathcal{B}$. Then
	\[ \mathcal{A} \pureindep[\mathcal{C}] \mathcal{B} \;\; \Leftrightarrow \;\; \forall a \in A \text{ and } b \in B \, \bigg[ \; a \leq b \Rightarrow \exists c \in C \; a \leq c \leq b \; \bigg]. \]			
\end{lemma}	

	\begin{proof} See~\cite{aviles}.	
\end{proof}

	Given $\mathcal{A} \models \mathrm{BA}$, we denote by $\mathrm{At}(\mathcal{A})$ the set of atoms of $\mathcal{A}$. We state few elementary facts about Boolean algebras that will be used freely throughout the paper.
	
	\begin{proposition} Let $\mathcal{A} \models \mathrm{BA}$ and $a \in A$. The following are equivalent.
		\begin{enumerate}[i)]
			\item $a$ is an atom.
			\item For every $b \in A$ either $a \leq b$ or $a \wedge b = 0$, but not both.
	\end{enumerate}		
\end{proposition}

	\begin{proposition} Let $\mathcal{A} \models \mathrm{BA}$ and $a \in A$. The following are equivalent.
		\begin{enumerate}[i)]
			\item $\mathcal{A}$ is atomic;
			\item Every $a \in A$ is the supremum of the atoms that it dominates.
			\item The unit is the supremum of the set of all atoms.
	\end{enumerate}		
\end{proposition}

	\begin{proposition}\label{distributivity} Let $\mathcal{A} \models \mathrm{BA}$, $a \in A$ and $(b_i)_{i \in I} \in A^I$. If $\bigvee_{i \in I} b_i$ exists, then
		\[ a \wedge \bigvee_{i \in I} b_i = \bigvee_{i \in I}(a \wedge b_i). \]
In particular, the supremum on the right-hand side of the equation exists.	
\end{proposition}

	\begin{notation}In Proposition \ref{distributivity}, we considered sups of families of points of a Boolean algebras without specifying the algebra with respect to which they were considered. This lack of notation may lead to some problems, in fact given $\mathcal{C} \leq \mathcal{A}$ and $(c_j)_{j \in J} \in C^J$ there may be a point $c \in C$ such that $c$ is the sup of the $c_j$ with respect to the algebra $\mathcal{C}$ but not with respect to the algebra $\mathcal{A}$. When we want to make clear with respect to which algebra we are considering a sup (resp. an inf) we use the notation $(\bigvee_{j \in J} c_j)^{\mathcal{A}}$ (resp. $(\bigwedge_{j \in J} c_j)^{\mathcal{A}}$), where obviously the superscript indicates the algebra with respect to which the sup (resp. the inf) is taken. This notation is heavier and it will be avoided when possible, it is, though, of crucial importance in the following definitions.
\end{notation}

	\begin{definition}[Regular Subalgebra \cite{halmos}] Let $\mathcal{A} \models \mathrm{BA}$ and $\mathcal{C} \leq \mathcal{A}$. We say that $\mathcal{C}$ is a {regular subalgebra} of $\mathcal{A}$ if whenever $(c_i)_{i \in I} \in C^I$ and $(\bigvee_{i \in I} c_i)^{\mathcal{C}}$ exists, then $(\bigvee_{i \in I} c_i)^{\mathcal{C}} = (\bigvee_{i \in I} c_i)^{\mathcal{A}}$.
\end{definition}

	\begin{definition} Let $\mathcal{A} \models \mathrm{BA}$ and $\mathcal{C} \leq \mathcal{A}$. We say that $\mathcal{C}$ sits nicely in $\mathcal{A}$, in symbols $\mathcal{C} \leq^* \mathcal{A}$, if the following conditions are satisfied:
		\begin{enumerate}[i)]
			\item $\mathcal{C}$ is a regular subalgebra of $\mathcal{A}$;
			\item for every $a \in A$ there exists a least $c \in C$ such that $a \leq c$.
		\end{enumerate}
\end{definition}

	Under the assumptions of atomicity of $\mathcal{C}$ and that $\mathcal{C}$ sits nicely in $\mathcal{A}$ and $\mathcal{B}$ we can give a characterization of the relation $\mathcal{A} \pureindep[\mathcal{C}] \mathcal{B}$ which will be very useful in the proof of the main theorem. The proof of this characterization is fairly standard, we include it for completeness of exposition.

	\begin{lemma}\label{referee_lemma} Let $\mathcal{D} \models \mathrm{BA}$, $\mathcal{A}, \mathcal{B} \leq \mathcal{D}$ and $\mathcal{C} \leq^* \mathcal{A}, \mathcal{B}$, with $\mathcal{C}$ atomic. Then 
	
			\[ \begin{array}{rcl}
	\mathcal{A} \displaystyle \pureindep[\mathcal{C}] \mathcal{B}	& \Leftrightarrow & \;\;\;\;\;\;\;\;\;\;\;\;\;\;\;\;  \forall c \in \mathrm{At}(\mathcal{C}) \text{, } a \in A \text{ and } b \in B \\
     	    &  & \bigg[ \; a, b \leq c \text{ and } a, b \not\in \left\{ 0, c \right\} \Rightarrow a \text{ incomparable to } b \, \big] \; \bigg]. 

\end{array} \]		
\end{lemma}

	\begin{proof} For the direction $(\Rightarrow)$, suppose there are $c \in \mathrm{At}(\mathcal{C})$, $a \in A$ and $b \in B$, such that $a, b \leq c$, $a, b \not\in \left\{ 0, c \right\}$, but $a \leq b$ or $b \leq a$. Suppose $a \leq b$, then it is impossible to find $c' \in C$ such that $0 < a \leq c' \leq b < c$, because $c$ is an atom. Suppose $b \leq a$, then $\neg a \leq \neg b$, and so again it is impossible to find $c' \in C$ such that $\neg c < \neg a \leq c' \leq \neg b < 1$, because $c$ is an atom.
	
	For the direction $(\Leftarrow)$, assume the right hand side and let $a \in A$ and $b \in B$, with $a \leq b$. Given that $\mathcal{C} \leq^* \mathcal{A}$, there is a least $c \in C$ such that $a \leq c$. We show that $c \leq b$. By atomicity of $\mathcal{C}$, we can find distinct $(c_i)_{i \in I} \in \mathrm{At}(\mathcal{C})^I$ such that $c = (\bigvee_{i \in I} c_i)^{\mathcal{C}}$, and, by the regularity assumption, $(\bigvee_{i \in I} c_i)^{\mathcal{C}} = (\bigvee_{i \in I} c_i)^{\mathcal{A}} = (\bigvee_{i \in I} c_i)^{\mathcal{B}}$. It suffices to show that for every $j \in I$ we have that $c_j \leq b$, because then $c = (\bigvee_{i \in I} c_i)^{\mathcal{C}} = (\bigvee_{i \in I} c_i)^{\mathcal{B}} \leq b$. Let then $j \in I$. By assumption $a \leq b$, so we have that $a \wedge c_j \leq b \wedge c_j \leq c_j$ and then, by the right hand side, $a \wedge c_j \in \left\{ 0, c_j \right\}$ or $b \wedge c_j \in \left\{ 0, c_j \right\}$. Now, if $a \wedge c_j \in \left\{ 0, c_j \right\}$, then $a \leq (\bigvee_{i \in I - \left\{ j \right\}} c_i)^\mathcal{C} = (\bigvee_{i \in I - \left\{ j \right\}} c_i)^\mathcal{A} < c$, contradicting the fact that $c$ is least. On the other hand, if $b \wedge c_j = 0$, then $a \wedge b \wedge c_j = a \wedge c_j = 0$, which as already been shown impossible. Then, we must have that $b \wedge c_j = c_j$, i.e. $c_j \leq b$, as wanted.

\end{proof}

	If on top of the assumptions of the previous lemma we further assume that $\mathcal{A}$ and $\mathcal{B}$ are atomic we can improve our characterization of independence of Boolean algebras. This characterization will not be needed in the proof of the main theorem, but it will simplify the arguments in Section \ref{sec_on_red}. We include it here for coherence of exposition. We first need a couple of lemmas, whose proofs are also standard.

\begin{lemma}\label{unicity_of_atom} Let $\mathcal{A} \models \mathrm{BA}$ and $\mathcal{C} \leq \mathcal{A}$, with $\mathcal{C}$ atomic and regular in $\mathcal{A}$. Then, for every $a \in \mathrm{At}(\mathcal{A})$ there exists exactly one $c \in \mathrm{At}(\mathcal{C})$ such that $a \leq c$.	
\end{lemma}

	\begin{proof} Let $a \in \mathrm{At}(\mathcal{A})$ and suppose that there are $c_0, c_1 \in \mathrm{At}(\mathcal{C})$ such that $a \leq c_0, c_1$ and $c_0 \neq c_1$. Then $a \leq c_0 \wedge c_1 = 0$, a contradiction. Suppose now that for every $c \in \mathrm{At}(\mathcal{C})$ we have that $a \nleq c$. Then $a \wedge c = 0$ for every $c \in \mathrm{At}(\mathcal{C})$, because $a$ is an atom of $\mathcal{A}$. Thus,
		\[a = a \wedge 1 = a \wedge (\bigvee_{c \in \mathrm{At}(\mathcal{C})} c)^\mathcal{C} = a \wedge (\bigvee_{c \in \mathrm{At}(\mathcal{C})} c)^\mathcal{A} = (\bigvee_{c \in \mathrm{At}(\mathcal{C})} (a \wedge c))^{\mathcal{A}} = 0, \]
a contradiction.
\end{proof}

	By Lemma~\ref{unicity_of_atom} whenever we have $\mathcal{A} \models \mathrm{BA}$ and $\mathcal{C} \leq \mathcal{A}$, with $\mathcal{C}$ atomic, we have a function $f: \mathrm{At}(\mathcal{A}) \rightarrow \mathrm{At}(\mathcal{C})$ which sends an atom $a$ to the unique atom $c$ that dominates $a$. We denote $f(a)$ by $\mathcal{C}^{\geq a}\!\! \downarrow$. 

	\begin{lemma}\label{char_indep_with_nice_alg} Let $\mathcal{D} \models \mathrm{BA}$, $\mathcal{A}, \mathcal{B} \leq \mathcal{D}$ and $\mathcal{C} \leq^* \mathcal{A}, \mathcal{B}$, with $\mathcal{A}, \mathcal{B}$ and $\mathcal{C}$ atomic. Then
		\[ \begin{array}{rcl}
	\mathcal{A} \displaystyle \pureindep[\mathcal{C}] \mathcal{B}	& \Leftrightarrow & \;\;\;\;\;\;\;\;\;\;\;\;\;\;\;\;  \forall a \in \mathrm{At}(\mathcal{A}) \text{ and } b \in \mathrm{At}(\mathcal{B}) \\
     	    &  & \bigg[ \; a \leq \neg b \Rightarrow \exists c \in \mathrm{At}(\mathcal{C}) \, \big[ \, a \leq \neg c \text{ and } b \leq c \, \big] \; \bigg]. 

\end{array} \]		

\end{lemma}

	\begin{proof} For the direction $(\Rightarrow)$, let $a \in \mathrm{At}(\mathcal{A})$ and $b \in \mathrm{At}(\mathcal{B})$, and suppose that $a \wedge b = 0$. By hypothesis there is $c \in C$ such that $a \wedge c = 0$ and $b \wedge \neg c = 0$. Let $\mathcal{C}^{\geq b}\!\! \downarrow = c^*$. Notice that $c \wedge c^* \neq 0$ because otherwise $b = 0$, given that $b \leq c, c^*$. But then $c^* \leq c$ because $c^* \in \mathrm{At}(\mathcal{C})$ and $c \in C$. Hence, $a \leq \neg c^*$ and $b \leq c^*$ because $a \leq \neg c \leq \neg c^*$.
		
		For the direction $(\Leftarrow)$, let $a \in A$ and $b \in B$, and suppose that $a \wedge b = 0$. If $a = 0$ or $b = 0$, we can find the wanted $c$, in the first case $c = 1$ and in the second $c = 0$. Suppose that $a, b \neq 0$, and let $(a_i)_{i \in I} \in \mathrm{At}(\mathcal{A})^I$ and $(b_j)_{j \in J} \in \mathrm{At}(\mathcal{B})^J$ such that 
		\[ a = \bigvee_{i \in I} a_i \; \text{ and } \; b = \bigvee_{j \in J} b_j. \]
Notice that $a_i \wedge b_j = 0$ for every $i \in I$ and $j \in J$, because $a \wedge b = 0$. Thus, by assumption, for every $i \in I$ and $j \in J$ there is $c_{i, j} \in \mathrm{At}(\mathcal{C})$ such that
		\[ a_i \wedge c_{i, j} = 0 \; \text{ and } \; b_j \wedge \neg c_{i, j} = 0. \]
For $j \in J$, let $c_j^{*} = \mathcal{C}^{\geq b_j}\!\! \downarrow$. Then for every $i \in I$ and $j \in J$ we have that $c_{i, j} = c_j^{*}$. And so for every $i \in I$ and $j \in J$ we have that
	\[ \neg c_j^{*} = \neg c_{i, j} \geq a_i \; \text{ and } c_j^{*} = c_{i, j} \geq b_j. \]
Define $c = (\bigvee_{j \in J} c_j^{*})^{\mathcal{C}}$. Notice that $c$ exists, and also 
	\[(\bigvee_{j \in J} c_j^{*})^{\mathcal{B}} = c = (\bigvee_{j \in J} c_j^{*})^{\mathcal{A}} \]
because by hypothesis $\mathcal{C} \leq^* \mathcal{A}, \mathcal{B}$. We claim that i)$\,c \geq b$ and ii)$\,\neg c \geq a$. To see i), notice that for every $j \in J$ we have that $c_j^* \geq b_j$ and so $c = \bigvee_{j \in J} c_j^* \geq \bigvee_{j \in J} b_j = b$. To see ii), notice that $\neg c = \bigwedge_{j \in J} \neg c_j^*$ and that $\neg c_j^* \geq a$ for every $j \in J$, because for every $i \in I$ we have that $\neg c_j^* \geq a_i$ and $a = \bigvee_{i \in I} a_i$.	
\end{proof}

	Finally, we can state our improved characterization of $\mathcal{A} \pureindep[\mathcal{C}] \mathcal{B}$. 

	\begin{corollary}\label{cor_for_indep_tuples} Let $\mathcal{D} \models \mathrm{BA}$, $\mathcal{A}, \mathcal{B} \leq \mathcal{D}$ and $\mathcal{C} \leq^* \mathcal{A}, \mathcal{B}$, with $\mathcal{A}, \mathcal{B}$ and $\mathcal{C}$ atomic. Then		
		\[ \mathcal{A} \pureindep[\mathcal{C}] \mathcal{B} \;\; \Leftrightarrow\;\; \forall a \in \mathrm{At}(\mathcal{A}) \text{ and } b \in \mathrm{At}(\mathcal{B}) \, \bigg[ \; a \leq \mathcal{C}^{\geq b} \!\! \downarrow \; \Rightarrow a \wedge b \neq 0 \; \bigg]. \]	
\end{corollary}

	\begin{proof} 
	For the direction $(\Leftarrow)$, let $a \in \mathrm{At}(\mathcal{A})$, $b \in \mathrm{At}(\mathcal{B})$ and $c = \mathcal{C}^{\geq b} \!\! \downarrow$, and suppose that $a \leq \neg b$, i.e. $a \wedge b = 0$. Then, by assumption, we have that $a \nleq c$, and so $a \wedge c = 0$, because $a \in \mathrm{At}(\mathcal{A})$ and $c \in C \subseteq A$. Thus, $b \leq c = \mathcal{C}^{\geq b} \!\! \downarrow$ and $a \leq \neg c$. Hence, by Lemma \ref{char_indep_with_nice_alg} we have that $\mathcal{A} \pureindep[\mathcal{C}] \mathcal{B}$.

	For the direction $(\Rightarrow)$, let again $a \in \mathrm{At}(\mathcal{A})$ and $b \in \mathrm{At}(\mathcal{B})$, and suppose that $a \wedge b = 0$. Then, by assumption, there exists $c \in \mathrm{At}(\mathcal{C})$ such that $a \wedge c = 0$ and $b \wedge \neg c = 0$. But if this is the case, then $b \leq c$ and so $c = \mathcal{C}^{\geq b} \!\! \downarrow$. Thus, $a \nleq \mathcal{C}^{\geq b} \!\! \downarrow$ because $a \wedge c = 0$ and $a \in \mathrm{At}(\mathcal{A})$.
		
\end{proof}

\begin{remark} Notice that in Lemma~\ref{char_indep_with_nice_alg} the assumption $\mathcal{C} \leq^* \mathcal{A}, \mathcal{B}$ is necessary. To see this, consider the following subalgebras of $\mathcal{P}(\omega)$
	\[ \mathcal{A} = \overline{\langle \left\{ \left\{ 2n, 2n + 2 \right\}\, | \, n \in 2\omega \right\} \cup \langle \left\{ \left\{ 2n + 1 \right\} \, | \, n \in \omega \right\}  \rangle}, \]
	\[ \mathcal{C} = \langle \left\{ \left\{ 2n + 1, 2n + 3 \right\}\, | \, n \in 2\omega \right\} \cup \langle \left\{ \left\{ 2n, 2n + 2 \right\} \, | \, n \in 2\omega \right\}  \rangle, \]
	\[ \mathcal{B} = \overline{\langle \left\{ \left\{ 2n + 1, 2n + 3 \right\}\, | \, n \in 2\omega \right\} \cup \langle \left\{ \left\{ 2n \right\} \, | \, n \in \omega \right\}  \rangle}. \]
where for $\mathcal{D} \models \mathrm{BA}$ with $\mathcal{D}$ complete and $\mathcal{E} \leq \mathcal{D}$ we denote by $\overline{\mathcal{E}}$ the completion of $\mathcal{E}$ in $\mathcal{D}$, and, as usual with fields of sets, we identify structures with their domains. Then $\mathcal{C} \leq \mathcal{A}, \mathcal{B}$ and  for all $a \in \mathrm{At}(\mathcal{A})$ and $b \in \mathrm{At}(\mathcal{B})$ whenever $a \cap b = \emptyset$ there is $c \in \mathrm{At}(\mathcal{C})$ such that $a \cap c = \emptyset$ and $b \cap \neg c = \emptyset$. But $\mathcal{A} \not\!\pureindep[\mathcal{C}] \mathcal{B}$ because for 
	\[ a = \left\{ 2n \, | \, n \in \omega \right\} \text{ and } b = \left\{ 2n + 1 \, | \, n \in \omega \right\}\]
we have that $a \cap b = \emptyset$, but for every $c \in \mathcal{C}$ if $b \subseteq c$ then $a \cap c \neq \emptyset$. \end{remark}

	We conclude this section with a basic lemma about atomless Boolean algebras which will be relevant in the proof of the main theorem.

	\begin{lemma}\label{splitting_atomless} Let $\mathcal{A} \models \mathrm{ABA}$. Then for every $a \in A$ with $a \neq 0$ there exists $(p_n(a))_{n < \omega} \in A^{\omega}$ such that for every $k < \omega$ the following holds:
		\begin{enumerate}[i)]
			\item $p_k(a) \neq 0$;
			\item $p_k(a) \wedge \bigvee_{i < k} p_i(a) = 0$;
			\item $\bigvee_{i \leq k}p_i(a) \lneq a$.
	\end{enumerate} 		
\end{lemma}



\subsection{Main Theorem}\label{main_theorem}

	The various lemmas proved in the previous section allowed us to characterize the relation $\mathcal{A} \pureindep[\mathcal{C}] \mathcal{B}$ under the assumption of atomicity of $\mathcal{C}$ and that it sits nicely in $\mathcal{A}$ and $\mathcal{B}$. We now use this characterization to establish a reduction of the relation $\mathcal{A} \pureindep[\mathcal{C}] \mathcal{B}$ to dividing in $\mathrm{ABA}$. This is the content of the following theorem.

	\begin{theorem}[Main Theorem]\label{char_div} Let $\mathfrak{M}$ be the monster model of $\mathrm{ABA}$, and $\mathcal{A}, \mathcal{B}$, $\mathcal{C} \leq \mathfrak{M}$, with $\mathcal{C}$ atomic and $\mathcal{C} \leq^* \mathcal{A}, \mathcal{B}$. Then
		
		\[ \mathcal{A} \pureindep[\mathcal{C}] \mathcal{B} \;\; \Leftrightarrow \;\; \forall a \in A \text{ and } b \in B \;\; \bigg[ \; a \displaydivindep[\mathcal{C}] b \; \bigg]. \]	
\end{theorem}

	\begin{proof} For the direction $(\Leftarrow)$, suppose that $a, b, c$ are a witness for $\mathcal{A} \not\!\pureindep[\mathcal{C}] \mathcal{B}$ as in Lemma~\ref{referee_lemma}. Possibly replacing $a$ with $c \wedge \neg a$ and $b$ with $c \wedge \neg b$, we may assume that $a \leq b$. Since $b < c$, using Lemma~\ref{splitting_atomless}, we can find $(b_i)_{i < \omega}$ indiscernible over $C$ such that $b_0 = b$ and $b_i \wedge b_j = 0$ for every $i < j < \omega$, witnessing that $\mathrm{tp}(a/Cb)$ divides over $C$.
	
	For the direction $(\Rightarrow)$, suppose that $\mathcal{A} \pureindep[\mathcal{C}] \mathcal{B}$ and let $a \in A$, $b \in B$ and $(b_i)_{i < \omega}$ indiscernible over $C$ such that $b_0 = b$. Let $c \in \mathrm{At}(\mathcal{C})$. If $a \wedge c \in \left\{ 0, c \right\}$ or $b \wedge c \in \left\{ 0, c \right\}$, let $a_c = a \wedge c$. Otherwise, $a \wedge c$ and $b \wedge c$ are incomparable. Choose $a_c \leq c$ so that $a_c \wedge b_i \neq 0$ but $a_c \lneq b$ for every $i < \omega$, this is possible because of Lemma~\ref{splitting_atomless} and the saturation of the monster model with respect to small sets. Thus, we have that $\mathrm{tp}(a_c/Cb_i) = \mathrm{tp}(a \wedge c/Cb)$ for every $i < \omega$. Again, by saturation, let $a'$ be such that $a' \wedge c = a_c$ for every $c \in \mathrm{At}(\mathcal{C})$. By quantifier elimination in $\mathrm{ABA}$ we have that $\mathrm{tp}(a'/Cb_i) = \mathrm{tp}(a/Cb)$ and $(b_i)_{i < \omega}$ is indiscernible over $Ca'$.
	
\end{proof}

	We then saw that under the assumptions $\mathcal{C} \leq^* \mathcal{A}, \mathcal{B}$ and $\mathcal{C}$ atomic, the relation $\mathcal{A} \pureindep[\mathcal{C}] \mathcal{B}$ is reducible to dividing in $\mathrm{ABA}$. At this point one may wonder: why $\mathrm{ABA}$? Can we reduce this relation to dividing in another completion of $\mathrm{BA}$? The first theory that comes to mind is the theory of infinite atomic Boolean algebras ($\mathrm{ATBA}$), which is also complete. The following remark shows that the reduction fails when dividing is considered in $\mathrm{ATBA}$.

	\begin{remark}\label{remark_on_atom} Let $\mathfrak{M}$ be the monster model of $\mathrm{ATBA}$, $(c, d, e, f) \in \mathrm{At}(\mathfrak{M})^4$ injective, $a = c \vee d$, $b = c \vee e$, $\mathcal{A} = \langle a \rangle$, $\mathfrak{B} = \langle b \rangle$ and $2 = \langle \emptyset \rangle$. Then $\mathcal{A} \pureindep[2] \mathcal{B}$ because $a \wedge b = c, a \wedge \neg b = d, \neg a \wedge b = e$ and $f \leq \neg a \wedge \neg b$. On the other hand, $a \not\!\divindep[2] b$. Let indeed $(p_i)_{i < \omega} \in (\mathrm{At}(\mathfrak{M}) - \left\{ c, e \right\})^{\omega}$ injective and define $b_0 = b$ and $b_i = p_{2i} \vee p_{2i +1}$ for $0 < i < \omega$. Then $(b_i)_{i < \omega}$ is indiscernible over $\emptyset$. Let now $a' \in \mathfrak{M}$ with $\mathrm{tp}(a'/b) = \mathrm{tp}(a/b)$, then $a' \wedge b \neq 0$ and $a' = q_0 \vee q_1$ for $q_0, q_1 \in \mathrm{At}(\mathfrak{M})$. Let now $0 < i < \omega$ be such that $p_{2i} \neq q_0, q_1$ and $p_{2i +1} \neq q_0, q_1$, then $a' \wedge b_i = 0$. Thus $(b_i)_{i < \omega}$ is not indiscernible over $a'$.
	
\end{remark}

\subsection{Reduction}\label{sec_on_red} 

	In this section we first introduce the notion of {\em algebra generated by a tuple of functions} and study some of its elementary properties. We then define an independence relation between tuples of functions, which is just an abstraction of the way independence is defined in database theory. Finally, we show that this form of independence is reducible to the independence relation between Boolean algebras that we dealt with in the previous sections.

	As usual, by field of sets on $M$ we mean a Boolean algebra of subsets of $M$.

	\begin{definition} Let $M$ be a set and $\mathcal{F}$ a field of sets on $M$. We define the $n$-th product of $\mathcal{F}$, in symbols $(\mathcal{F})^n$, to be the field of sets on $M^n$ generated by the family of sets $\left\{ \prod_{i < n} A_i \, | \, (A_i)_{i < n} \in \mathcal{F}^n \right\}$.
\end{definition}

	\begin{proposition} Let $M$ be a set and $\mathcal{F}$ a field of sets on $M$. Then
		\[ (\mathcal{F})^n = \mbox{\Large $\{$ } \!\! \bigcup_{i < k} (\prod_{j < n} A_{i, j}) \, | \, (A_{i, j})_{\substack{i < k \\ j < n} } \in \mathcal{F}^{nk}, k < \omega \mbox{\Large $\}$}.\]
\end{proposition}

	As known, given two sets $I$ and $M$, a field of sets on $M$ and a function $f: I \rightarrow M$, the family of sets $\mbox{\Large $\{$ } \!\! f^{(-1)}(A)\, | \, A \in \mathcal{F} \mbox{\Large $\}$}$ is a field of sets on $I$. In the following, given a tuple of functions $(f_i)_{i < n}$ from $I$ to $M$, we will identify the tuple $(f_i)_{i < n}$ with the function $f: I \rightarrow M^n$ such that $f(j) = (f_0(j), ..., f_{n-1}(j))$ for every $j \in I$.

	\begin{definition} Let $f = (f_i)_{i < n}$ be a tuple of functions from $I$ to $M$, and $\mathcal{F}$ the finite-cofinite field of sets on $M$. We let the Boolean algebra generated by $f$, in symbols $\pi(f)$, to be the field of sets $\mbox{\Large $\{$ } \!\! f^{(-1)}(A)\, | \, A \in (\mathcal{F})^n \mbox{\Large $\}$} \subseteq \mathcal{P}(I)$.		
\end{definition}

	\begin{proposition} Let $f = (f_i)_{i < n}$ be a tuple of functions from $I$ to $M$. Then
		\begin{enumerate}[i)]
			\item $\pi(f) = \langle \pi(f_0) \cup \cdots \cup \pi(f_{n-1}) \rangle = \langle f_i^{(-1)}(a) \, | \, i < n, a \in \mathrm{ran} (f_i) \rangle$;
			\item $\mathrm{At}(f) = \mbox{\Large $\{$ } \!\! f^{(-1)}(a) \, | \, a \in \mathrm{ran}(f) \mbox{\Large $\}$}$;
			\item $\pi(f)$ is atomic.
	\end{enumerate}
\end{proposition} 

	\begin{lemma} Let $f = (f_j)_{j < n}$ be a tuple of functions from $I$ to $M$. Then $\pi(f)$ is a regular subalgebra of $\mathcal{P}(I)$.
		
\end{lemma}

	\begin{proof} Suppose otherwise, then 
	we can find $(C_t)_{t \in T} \in \pi(f)^T$ such that
		\[ (\bigvee_{t \in T} C_t)^{ \pi(f)} = I \neq (\bigvee_{t \in T} C_t)^{ \mathcal{P}(I)} = \bigcup_{t \in T} C_t, \]
and so there exists $i \in I$ such that $i \notin C_t$ for every $t \in T$. Let $C = f^{(-1)}(f(i))$. If there were $t \in T$ such that $C \subseteq C_t$, then $i \in f^{(-1)}(f(i)) = C \subseteq \bigcup_{t \in T} C_t$, a contradiction. Thus, $C \cap C_t = \emptyset$ for every $t \in T$, because $C$ is an atom.  But then $C_t \subseteq C^0$ for every $t \in T$ and $C^0 \neq I$ because $C \neq \emptyset$, a contradiction.
		
\end{proof}

	\begin{corollary} Let $f = (f_j)_{j < n}$ and $h = (h_j)_{j < k}$ be tuples of functions from $I$ to $M$. Then 
		\begin{enumerate}[i)]
			\item $\pi(h)$ is a regular subalgebra of $\pi(fh)$;
			\item In both $\pi(h)$ and $\pi(fh)$ sups are unions.
		\end{enumerate}		
\end{corollary}

\begin{proof} Immediate from the coherence of the regular subalgebra relations, i.e. if $\mathcal{C} \leq \mathcal{B}$, $\mathcal{B}$ is regular in $\mathcal{A}$ and $\mathcal{C}$ is regular in $\mathcal{A}$, then $\mathcal{C}$ is regular in $\mathcal{B}$.

\end{proof}


	



	\begin{lemma} Let $f = (f_j)_{j < n}$ and $h = (h_j)_{j < k}$ be tuples of functions from $I$ to $M$. Then $\pi(h) \leq^* \pi(fh)$.
\end{lemma}

	\begin{proof} Let $(A_i)_{i \in I} \in \mathrm{At}(\pi(fh))^I$ and suppose that	$\bigcup_{i \in I} A_i = A$ exists in $\pi(fh)$. Let $C_i = \pi(h)^{\geq A_i}\!\! \downarrow$, we want to show that $\bigcup_{i \in I} C_i \in \pi(h)$. Given that $A \in \pi(fh)$, there are $(D_{i, j})_{\substack{i < t \\ j < n} } \in \mathcal{F}^{nt}$ and $(E_{i, j})_{\substack{i < t \\ j < k} } \in \mathcal{F}^{kt}$ such that
		\[ A = \bigcup_{i < t}((fh)^{(-1)}(\prod_{j < n} D_{i, j} \times \prod_{j < k} E_{i, j})) = (fh)^{(-1)}(\bigcup_{i < t}(\prod_{j < n} D_{i, j} \times \prod_{j < k} E_{i, j})) . \]
But then
		\[ \bigcup_{i \in I} C_i = \bigcup_{i < t}(h^{(-1)}(\prod_{j < k} E_{i, j})) \in \pi(h). \]
\end{proof}

	\begin{definition} Let $f = (f_i)_{i < n}$, $g = (g_i)_{i < m}$ and $h = (h_i)_{i < k}$ be tuples of functions from $I$ to $M$. We define
\[\begin{array}{rcl}
 & f \displaystyle \pureindep[h] g &  \\
 & \Leftrightarrow &  \\
 & \forall p, q \in I (h(p) = h(q) \Rightarrow \exists t \in I (h(t) = h(p) \, \& \, f(t) = f(p) \, \& \, g(t) = g(q))). & 
		
\end{array} \]
\end{definition}

	We now show that two tuples of functions are independent over a third if and only if the corresponding Boolean algebras are independent (in the sense of Definition \ref{independence} and Corollary \ref{cor_for_indep_tuples}). As we will see in Section~\ref{indep_in_dat_th}, the following theorem establishes a reduction of the form of independence at play in database theory and independence logic to dividing in $\mathrm{ABA}$. 
	
		\begin{theorem}\label{reduction} Let $f = (f_i)_{i < n}$, $g = (g_i)_{i < m}$ and $h = (h_i)_{i < k}$ be tuples of functions from $I$ to $M$. Then
		\[\begin{array}{rcl}
		 & f \displaystyle \pureindep[h] g &  \\
		 & \Leftrightarrow &  \\
		 & \pi(fh) \displaystyle \pureindep[\pi(h)] \pi(gh) . & 
				
\end{array} \]
\end{theorem}

\begin{proof} We use Corollary \ref{cor_for_indep_tuples}. For the direction $(\Leftarrow)$, let $p, q \in I$ with $\vec{h}(p) = \vec{h}(q)$. Let $\vec{d} \in M^n$, $\vec{e} \in M^m$ and $\vec{v} \in M^k$ be such that
	\[ \vec{d}\vec{v} = \vec{f}\vec{h}(p) \; \text{ and } \; \vec{e}\vec{v} = \vec{g}\vec{h}(q),\]
and let
	\[ A = (\vec{f}\vec{h})^{(-1)}(\vec{d}\vec{v}) \; \text{ and } \; B = (\vec{g}\vec{h})^{(-1)}(\vec{e}\vec{v}). \]
Notice that $A \in \mathrm{At}(\pi(\vec{f} \vec{h}))$ and $B \in \mathrm{At}(\pi(\vec{g} \vec{h}))$. Clearly $\pi(\vec{h})^{\geq B}\!\! \downarrow = \vec{h}^{(-1)}(\vec{v})$, and so $A \subseteq \pi(\vec{h})^{\geq B}\!\! \downarrow$. Thus, by hypothesis, $A \cap B \neq \emptyset$. Let $t \in A \cap B$, then
	\[ t \in (\vec{f}\vec{h})^{(-1)}(\vec{d}\vec{v}) \; \text{ and } \; t \in (\vec{g}\vec{h})^{(-1)}(\vec{e}\vec{v}). \]
And so 
\[ \vec{h}(t) = \vec{v} = \vec{h}(p) \, \text{ and } \, \vec{f}(t) = \vec{d} = \vec{f}(p) \, \text{ and } \, \vec{g}(t) = \vec{e} = \vec{g}(q)). \]

For the direction $(\Rightarrow)$, let $A \in \mathrm{At}(\pi(\vec{f} \vec{h}))$, $B \in \mathrm{At}(\pi(\vec{g} \vec{h}))$ and $C = \pi(\vec{h})^{\geq B}\!\! \downarrow$, and suppose that $A \subseteq C$. By our description of the atoms of the algebras in question, there are $\vec{d} \in M^n$, $\vec{e} \in M^m$ and $\vec{v}, \vec{w} \in M^k$ such that
\[ A = (\vec{f}\vec{h})^{(-1)} (\vec{d}\vec{v}) \; \text{ and } \; B = (\vec{g}\vec{h})^{(-1)} (\vec{e}\vec{w}). \]
Notice now that $C = \vec{h}^{(-1)}(\vec{w})$ and let $i \in (\vec{f}\vec{h})^{(-1)} (\vec{d}\vec{v})$, then $\vec{v} = \vec{h}(i) = \vec{w}$, because $(\vec{f}\vec{h})^{(-1)} (\vec{d}\vec{v}) \subseteq \vec{h}^{(-1)}(\vec{w})$. Let $p \in A$ and $q \in B$, then $\vec{h}(p) = \vec{h}(q)$ and so, by hypothesis, there exists $t \in I$ such that
\[ \vec{h}(t) = \vec{v} = \vec{h}(p) \, \text{ and } \, \vec{f}(t) = \vec{d} = \vec{f}(p) \, \text{ and } \, \vec{g}(t) = \vec{e} = \vec{g}(q). \]
But then $A \cap B \neq \emptyset$, because
	\[\begin{array}{rcl}
t \in (\vec{f}\vec{h})^{(-1)}(\vec{f}\vec{h}(t)) & = & (\vec{f}\vec{h})^{(-1)}(\vec{f}\vec{h}(p)) \\
 		& = & (\vec{f}\vec{h})^{(-1)}(\vec{d}\vec{v}) \\
		& = & A,
\end{array} \]
and 		
	\[\begin{array}{rcl}
t \in (\vec{g}\vec{h})^{(-1)}(\vec{g}\vec{h}(t)) & = & (\vec{g}\vec{h})^{(-1)}(\vec{g}(q)\vec{h}(p)) \\
		& = & (\vec{g}\vec{h})^{(-1)}(\vec{g}(q)\vec{h}(q))  \\
 		& = & (\vec{g}\vec{h})^{(-1)}(\vec{e}\vec{w})  \\
		& = & B.
\end{array} \]
\end{proof}

\section{Independence in Database Theory and Team Semantics}\label{indep_in_dat_th}

\subsection{Database Independence}\label{emb_mult_dep}

	Let $\mathrm{Var}$ be a countable set of symbols, called {\em attributes} or {\em individual variables}. A {\em relation schema} is a finite set $R = \left\{ x_{0}, ..., x_{n-1} \right\}$ of attributes from $\mathrm{Var}$. Each attribute $x_i$ of a relation schema is associated with a domain $\mathrm{dom}(x_i)$ which represents the set of possible values that can occur as values of $x_i$. A {\em tuple} over $R$ is a function $t: R \rightarrow \bigcup_{i < n} \mathrm{dom}(x_i)$ with $t(x_i) \in \mathrm{dom}(x_i)$, for all $i < n$. A {\em database}\footnote{In the context of database theory a database is usually taken to be a {\em finite} set of tuples over a relation schema $R$. In our framework the assumption of finiteness does not play any role and so we drop it.} $r$ over $R$ is a set of tuples over $R$. For $x \subseteq R$ and $t \in r$ we let $t(x)$ to be the restriction of the function $t$ to $x$.

	\begin{definition}[Functional dependence \cite{armstrong}] Let $R$ be a relation schema, $x$ and $y$ tuples of attributes from $R$, and $r$ a database over $R$. We define
		
		\[ r \text{ satisfies } x \rightarrow y \;\; \Leftrightarrow \;\; \forall t_0, t_1 \in r (t_0(x) = t_1(x) \Rightarrow t_0(y) = t_1(y)). \]
If $r$ satisfies $x \rightarrow y$ we say that $r$ manifests the {\em functional dependency} $x \rightarrow y$.
\end{definition}

	\begin{definition}[Database independence \cite{sagiv}\footnote{In standard references in database theory (among which \cite{sagiv}) it is usually assumed that $x \cap y \subseteq z$. We relax this assumption because it comes at no conceptual cost and simplifies the treatment.}] Let $R$ be a relation schema, $x$, $y$ and $z$ tuples of attributes from $R$, and $r$ a database over $R$. We define
		
		\[\begin{array}{rcl}
		 &  r \text{ satisfies } z \twoheadrightarrow x \, | \, y &  \\
		 & \Leftrightarrow &  \\
		 & \!\!\!\!\!\forall t_0, t_1 \in r (t_0(z) = t_1(z) \Rightarrow \exists t_2 \in r (t_2(z) = t_0(z) \, \& \, t_2(x) = t_0(x) \, \& \, t_2(y) = t_1(y))). & 

\end{array} \]
If $r$ satisfies $z \twoheadrightarrow x \, | \, y$ we say that $r$ manifests the {\em database independency} $z \twoheadrightarrow x \, | \, y$. 
	
\end{definition}

	In the database theory literature the term {\em embedded multivalued dependence} is preferred to the simpler database independence. The reason for this choice of terminology is that embedded multivalued dependence is a generalization of functional dependence, as the following proposition shows.

	\begin{proposition}\label{functional_dep_as_embed} Let $R$ be a relation schema, $x$ and $y$ tuples of attributes from $R$, and $r$ a database over $R$. Then
		\[ r \text{ satisfies }  x \rightarrow y \;\; \Leftrightarrow \;\; r \text{ satisfies }  x \twoheadrightarrow y \, | \, y \]
	
\end{proposition}

	\begin{proof} Immediate.
	
\end{proof}

	Let $R$ be a relation schema, $x$ a tuple of attributes from $R$ and $r$ a database over $R$. Let also $r = (t_i)_{i \in I}$ and $M = \bigcup_{i < \omega} \mathrm{dom}(x_i)$. We can then define a function $\dot{x}: I \rightarrow M^n$ by letting $\dot{x}(i) = t_i(x)$ for every $i \in I$. Clearly, for $x = (x_{j_0}, ..., x_{j_{m-1}})$ and $i \in I$, we have that $\dot{x}(i) = (\dot{x}_{j_0}(i), ..., \dot{x}_{j_{m-1}}(i))$, and so we can identify the objects $\dot{x}$ and $(\dot{x}_{j_0}, ..., \dot{x}_{j_{m-1}})$. Following the notation of Section~\ref{sec_on_red}, we can then consider the Boolean algebra $\pi(\dot{x})$ (as a subalgebra of $\mathcal{P}(I)$). Notice now that being the theory $\mathrm{ABA}$ the model completion of the theory $\mathrm{BA}$, there is an embedding $i: \pi(\dot{x}) \rightarrow \mathfrak{M}$, where as usual we denote by $\mathfrak{M}$ the monster model of $\mathrm{ABA}$. Thus $\pi(\dot{x}) \cong i(\pi(\dot{x}))$, and so, modulo isomorphism, the Boolean algebra $\pi(\dot{x})$ can be thought as living in $\mathfrak{M}$. This little argument allows us to formulate in exact terms a reduction of database independence to dividing in $\mathrm{ABA}$.  
	
	\begin{theorem} Let $R$ be a relation schema, $x$, $y$ and $z$ tuples of attributes from $R$, and $r$ a database over $R$. Then
		
		\[ r \text{ satisfies }  z \twoheadrightarrow x \, | \, y \;\; \Leftrightarrow \;\; \dot{x} \pureindep[\dot{z}] \dot{y} \;\; \Leftrightarrow \;\; \pi(\dot{x}\dot{z}) \pureindep[\pi(\dot{z})] \pi(\dot{y}\dot{z}) \]
		
		\[\Leftrightarrow \;\; \forall a \in \mathrm{At}(\pi(\dot{x}\dot{z})) \text{ and } b \in \mathrm{At}(\pi(\dot{y}\dot{z}))  \;\; \bigg[ \; a \displaydivindep[\pi(\dot{z}) ] b \; \bigg]. \]
	
\end{theorem}

	\begin{proof} An immediate consequence of Theorems~\ref{reduction} and \ref{char_div}.
	
\end{proof}

	As we saw above, in the case of $\mathrm{ABA}$ the choice of embeddings with respect to which we think of the algebras arising from a database as living in the monster model for $\mathrm{ABA}$ does not matter, any choice of embeddings would do. In the case of $\mathrm{ATBA}$ this is not the case. Remark~\ref{remark_on_atom} shows in fact that in the case of $\mathrm{ATBA}$ the choice of embeddings {\em does} matter. For this reason, $\mathrm{ABA}$ is a more natural context than $\mathrm{ATBA}$ for our reduction to take place.
	
\subsection{Independence in Team Semantics}\label{indep_team_sem}

	We now introduce what is known as independence logic \cite{GV12}. The semantics of this logic is formulated using sets of assignments, also called {\em teams}, instead of single assignments. This new way of defining semantics for logical languages was introduced by Hodges in \cite{hodges} and then developed by V\"a\"an\"anen in \cite{MR2351449}. The salient characteristic of independence logic is that it contains syntactic expressions modelling dependence and independence, called {\em dependence and independence atoms}, respectively. On top of this, the language of independence logic allows relational atomic formulas, Boolean connectives and quantifiers, as in classical first-order logic. We focus here only on (in)dependence atoms and consequently give the semantics only for this fragment of the language of this logic. The truth definition can then be extended from independence atoms to any formula in the language of independence logic in a canonical way. The resulting logic has a non-classical flavour and possesses the expressive power of existential second-order logic. The exact definitions are beyond the scope of this paper and will not be presented here, for details see \cite{MR2351449}.
	
	Let $\mathcal{M}$ be a first-order structure and $V \subseteq \mathrm{Var}$ a finite set of variables. An assignment $s$ on $\mathcal{M}$ with domain $\mathrm{dom}(s) = V$ is a mapping from $V$ to $M$. A team $X$ on $\mathcal{M}$ with domain $\mathrm{dom}(X) = V$ is a set of assignments with domain $V$. Let $L$ be a fixed but arbitrary signature. The dependence and independence atoms in the signature $L$ are syntactic expressions of the form $\dep(u, v)$ and 
$u \perpc{w} v$, respectively, where $u, v, w \in (\mathrm{Term}_{L})^{< \omega}$. 
	Given $u \in (\mathrm{Term}_{L})^{< \omega}$, we denote by $\mathrm{Var}(u)$ the set of variables occurring in at least one of the terms in the tuple $u$. 

	\begin{definition}\label{seman_indep_logic} Let $\mathcal{M}$ be an $L$-structure, $X$ a team on $\mathcal{M}$ and $\mathrm{Var}(u), \mathrm{Var}(v) \subseteq \mathrm{dom}(X)$. We let
		
		\[ \mathcal{M}\models_X \dep(u, v) \;\; \Leftrightarrow \;\; \forall s, s' \in X \, (s(u) = s'(u) \Rightarrow s(v) = s'(v)). \]
		
\end{definition}

	\begin{definition}  Let $\mathcal{M}$ be an $L$-structure, $X$ a team on $\mathcal{M}$ and $\mathrm{Var}(u), \mathrm{Var}(v)$ and $\mathrm{Var}(w) \subseteq \mathrm{dom}(X)$. We let
		
		\[\begin{array}{rcl}
		 & \mathcal{M} \models_X u \perpc{w} v &  \\
		 & \Leftrightarrow &  \\
		 & \!\!\!\!\!\forall s, s' \in X (s(w) = s'(w) \Rightarrow \exists s'' \in X (s''(w) =s(w) \&  s''(u) = s(u)  \&  s''(v) = s'(v))). & 

\end{array} \]
		
\end{definition}

	In analogy with Proposition~\ref{functional_dep_as_embed}, we have the following. It shows that our syntax could have been chosen without dependence atoms.

	\begin{proposition}\label{dep_atom_as_indep_atom} Let $\mathcal{M}$ be a first-order structure, $X$ a team on $\mathcal{M}$ and $\mathrm{Var}(u), \newline \mathrm{Var}(v) \subseteq \mathrm{dom}(X)$. Then
		
		\[ \mathcal{M} \models_X \dep(u, v) \;\; \Leftrightarrow \;\; \mathcal{M} \models_X v \perpc{u} v. \]
	
\end{proposition}

	Given a non-empty set $M$, we denote by $(M) = M$ the structure in the empty signature with domain $M$. Clearly, given $R = \left\{ x_{0}, ..., x_{n-1} \right\} \subseteq \mathrm{Var}$, each database $r$ over $R$ can be seen as a team $X_r$ on $\bigcup_{i < n} \mathrm{dom}(x_{i})$ with domain $R$, and each team $X$ on a substructure of $\bigcup_{i < n} \mathrm{dom}(x_{i})$ with domain $R$ can be seen as a database $r_X$ over $R$. 
	
	\begin{remark}\label{red_to_database} Let $M$ be a non-empty set, $X$ a team on $M$, $x, y$ and $z \subseteq \mathrm{dom}(X) \subseteq \mathrm{Var}$. Then 
		\[ M \models_X x \perpc{z} y \;\; \Leftrightarrow \;\; r_X \text{ satisfies } z \twoheadrightarrow x \, | \, y. \]
		
\end{remark}

	As Remark~\ref{red_to_database} shows and any researcher in dependence logic knows, the form of independence at play in team semantics is {\em exactly} database independence. Thus, as the latter case of independence is reducible to diving in $\mathrm{ABA}$, so is the former.
	
	As made clear by our exposition of the subject, independence logic is an extension of first-order logic which is based on a database-oriented notion of team, and a particular form of independence, i.e. database independence. In light of the analogies between this form of independence and stochastic independence noted above, it seems plausible to formulate a version of independence logic which admits probabilities and it is able to deal with stochastic independence. This is done by the authors in \cite{quantum_teams} and \cite{measure_teams}, where applications of these ideas to quantum logic are also developed.


\begin{thebibliography}{10}

	\bibitem{armstrong}
	William~Ward Armstrong.
	\newblock {\em Dependency {S}tructures of {D}ata {B}ase {R}elationships}.
	\newblock In IFIP Congress, pages 580-583, 1974.

	\bibitem{aviles}
	Antonio Avilés and Christina Brech.
	\newblock {\em A Boolean Algebra and a {B}anach Space Obtained by Push-Out
	  Iteration}.
	\newblock Topology Appl., 158(13):1534-1550, 2011.
	
	\bibitem{baldwin}
	John T. Baldwin, Paul C. Eklof and Jan Trlifaj.
	\newblock {\em $^{\perp}N$ as an Abstract Elementary Class}.
	\newblock Ann. Pure Appl. Logic, 149(1):25-39, 2007.


	\bibitem{on_ran_var}
	Ita{\"i} Ben~Yaacov.
	\newblock {\em On {T}heories of {R}andom {V}ariables}.
	\newblock Israel J. Math., 194(2):957-1012, 2013.

	\bibitem{CFOL_and_stab}
	Ita{\"i} Ben~Yaacov and Alexander Usvyatsov.
	\newblock {\em Continuous {F}irst {O}rder {L}ogic and {L}ocal {S}tability}.
	\newblock Trans. Amer. Math. Soc.,
	 362(10):5213-5259, 2010.

	\bibitem{halmos}
	Steven Givant and Paul Halmos.
	\newblock {\em Introduction to {B}oolean {A}lgebras}.
	\newblock Springer-Verlag, 2009.
	
	\bibitem{GV12}
	Erich Gr{\"a}del and Jouko V{\"a}{\"a}n{\"a}nen.
	\newblock {\em Dependence, {I}ndependence, and {I}ncomplete {I}nformation}.
	\newblock Studia Logica, 101(2):399-410, 2013.

	\bibitem{hodges}
	Wilfrid Hodges.
	\newblock {\em Compositional Semantics for a Logic of Imperfect Information}.
	\newblock Log. J. IGPL, 5:539-563, 1997.
	
	\bibitem{asa}
	{\AA}sa Hirvonen.
	\newblock {\em Independence in Model Theory}.
	\newblock In S. Abramsky, J. Kontinen, H. Vollmer and J. V\"a\"an\"anen, editors, Dependence Logic: Theory and Applications, Springer. To appear.
	
	\bibitem{asa_and_tapani}
	{\AA}sa Hirvonen and Tapani Hyttinen.
	\newblock {\em Categoricity in Homogeneous Complete Metric Spaces}.
	\newblock Arch. Math. Logic, 48:269-322, 2009.
	
	\bibitem{asa_and_tapani_padics}
	{\AA}sa Hirvonen and Tapani Hyttinen.
	\newblock {\em Measuring Dependence in Metric Abstract Elementary Classes with Perturbations}.
	\newblock Submitted.
	
	\bibitem{tapani_and_kaisa}
	Tapani Hyttinen and Kaisa Kangas.
	\newblock {\em On Model Theory of Covers of Algebraically Closed Fields}.
	\newblock Annales Academi{\ae} Scientiarum Fennic{\ae} Mathematica 40(2):507-533, 2015.
	

	\bibitem{measure_teams}
	Tapani Hyttinen, Gianluca Paolini and Jouko V{\"a}{\"a}n{\"a}nen.
	\newblock {\em A Logic for Arguing About Probabilities in Measure {T}eams}.
	\newblock Submitted.
	
	\bibitem{geom_lat}
	Tapani Hyttinen and Gianluca Paolini.
	\newblock {\em Beyond Abstract Elementary Classes: On The Model Theory of Geometric Lattices}.
	\newblock Submitted. 
	
	\bibitem{quantum_teams}
	Tapani Hyttinen, Gianluca Paolini, and Jouko V{\"a}{\"a}n{\"a}nen.
	\newblock {\em {Q}uantum {T}eam {L}ogic and Bell's Inequalities}.
	\newblock Review of Symbolic Logic, 08(04):722-742, 2015.

	\bibitem{marker}
	David Marker.
	\newblock {\em Introduction to Model Theory}.
	\newblock Springer, 2002.
	
	\bibitem{pearl}
	Judea Pearl.
	\newblock {\em Probabilistic Reasoning in Intelligent Systems: Networks of Plausible Inference}.
	\newblock Morgan Kaufman, San Mateo CA, 1988.

	\bibitem{poizat}
	Bruno Poizat.
	\newblock {\em A {C}ourse in {M}odel {T}heory}.
	\newblock Springer-Verlag, 2000.

	\bibitem{sagiv}
	Yehoshua Sagiv and Scott~F. Walecka.
	\newblock {\em Subset {D}ependencies and a {C}ompleteness {R}esult for a {S}ubclass
	  of {E}mbedded {M}ultivalued {D}ependencies}.
	\newblock J. ACM, 29(1):103-117, 1982.

	\bibitem{shelah}
	Saharon Shelah.
	\newblock {\em Classification Theory: and the Number of Non-Isomorphic Models}.
	\newblock Elsevier, 1990.
	
	\bibitem{studeny}
	Milan Studeny.
	\newblock {\em Conditional Independence Relations Have No Finite Complete Characterization}.
	\newblock In: Transactions of the 11th Prague Conference on Information Theory,
pp. 377-396. Kluwer, 1992.

	\bibitem{MR2351449}
	Jouko V{\"a}{\"a}n{\"a}nen.
	\newblock {\em Dependence logic}, volume~70 of {\em London Mathematical Society
	  Student Texts}.
	\newblock Cambridge University Press, Cambridge, 2007.

	\end{thebibliography}
\end{document}